\documentclass[reqno,a4paper]{amsart}%
\usepackage{amssymb}
\usepackage[colorlinks=true,hypertex]{hyperref}
\allowdisplaybreaks[4]
\theoremstyle{plain}
\newtheorem{thm}{Theorem}

\newtheorem{open}{Open Problem}
\theoremstyle{remark}
\newtheorem{rem}{Remark}

\date{Commenced on 27 January 2009 and completed on 17 February 2009 in Melbourne; Revised on 31 March 2009 and 3 July 2009 in Jiaozuo}
\date{}

\begin{document}

\title[Sharpening and generalizations of Carlson's double inequality]
{Sharpening and generalizations of Carlson's inequality for the arc cosine function}

\author[F. Qi]{Feng Qi}
\address[F. Qi]{Department of Mathematics, College of Science, Tianjin Polytechnic University, Tianjin City, 300160, China}
\email{\href{mailto: F. Qi <qifeng618@gmail.com>}{qifeng618@gmail.com}, \href{mailto: F. Qi <qifeng618@hotmail.com>}{qifeng618@hotmail.com}, \href{mailto: F. Qi <qifeng618@qq.com>}{qifeng618@qq.com}}
\urladdr{\url{http://qifeng618.spaces.live.com}}

\author[B.-N. Guo]{Bai-Ni Guo}
\address[B.-N. Guo]{School of Mathematics and Informatics, Henan Polytechnic University, Jiaozuo City, Henan Province, 454010, China}
\email{\href{mailto: B.-N. Guo <bai.ni.guo@gmail.com>}{bai.ni.guo@gmail.com}, \href{mailto: B.-N. Guo <bai.ni.guo@hotmail.com>}{bai.ni.guo@hotmail.com}}
\urladdr{\url{http://guobaini.spaces.live.com}}

\begin{abstract}
In this paper, we sharpen and generalize Carlson's double inequality for the arc cosine function.
\end{abstract}



\thanks{The first author was partially supported by the China Scholarship Council}

\thanks{This paper was typeset using \AmS-\LaTeX}

\maketitle

\section{Introduction and main results}

In~\cite[p.~700, (1.14)]{Carlson-pams-ellip} and~\cite[p.~246, 3.4.30]{mit}, it was listed that
\begin{equation}\label{Carlson's-ineq-arccos}
\frac{6(1-x)^{1/2}}{2\sqrt2\,+(1+x)^{1/2}}<\arccos x<\frac{\sqrt[3]4\,(1-x)^{1/2}}{(1+x)^{1/6}}, \quad 0\le x<1.
\end{equation}
\par
In \cite{Carlson-Arccos.tex}, the right-hand side inequality in~\eqref{Carlson's-ineq-arccos} was sharpened and generalized.
\par
On the other hand, the left-hand side inequality in~\eqref{Carlson's-ineq-arccos} was also generalized slightly in \cite{Carlson-Arccos.tex} as follows: For $x\in(0,1)$, the function
\begin{equation}
F_{1/2,1/2,2\sqrt2\,}(x)=\frac{2\sqrt2\,+(1+x)^{1/2}}{(1-x)^{1/2}}\arccos x
\end{equation}
is strictly decreasing. Consequently, the double inequality
\begin{equation}\label{Carlson-Arccos-ineq-2}
\frac{6(1-x)^{1/2}}{2\sqrt2\,+(1+x)^{1/2}}<\arccos x<\frac{\bigl(1/2+\sqrt2\,\bigr)\pi(1-x)^{1/2}}{2\sqrt2\,+(1+x)^{1/2}}
\end{equation}
holds on $(0,1)$ and the constants $6$ and $\bigl(\frac12+\sqrt2\,\bigr)\pi$ are the best possible.
\par
The aim of this paper is to further generalize the left-hand side inequality in~\eqref{Carlson's-ineq-arccos}.
\par
Our main results may be stated as follows.

\begin{thm}\label{Carlson-Arccos-thm-2}
Let $a$ be a real number and
\begin{equation}
F_{a}(x)=\frac{a+(1+x)^{1/2}}{(1-x)^{1/2}}\arccos x, \quad x\in(0,1).
\end{equation}
\begin{enumerate}
\item
If $a\le\frac{2(\pi-2)}{4-\pi}$, the function $F_a(x)$ is strictly increasing;
\item
If $a\ge2\sqrt2\,$, then the function $F_a(x)$ is strictly decreasing;
\item
If $\frac{2(\pi-2)}{4-\pi}<a<2\sqrt2\,$, the function $F_a(x)$ has a unique minimum.
\end{enumerate}
\end{thm}

\begin{thm}\label{Carlson-Arccos-thm-3}
For $a\le\frac{2(\pi-2)}{4-\pi}$,
\begin{equation}\label{Carlson-Arccos-thm-3-ineq}
\frac{[\pi(1+a)/2](1-x)^{1/2}}{a+(1+x)^{1/2}}<\arccos x <\frac{\bigl(2+\sqrt2\,a\bigr)(1-x)^{1/2}}{a+(1+x)^{1/2}},\quad x\in(0,1).
\end{equation}
For $\frac{2(\pi-2)}{4-\pi}<a<2\sqrt2\,$,
\begin{multline}\label{Carlson-Arccos-thm-3-ineq-max}
\frac{8(1-2/a^2)(1-x)^{1/2}}{a+(1+x)^{1/2}}<\arccos x\\*
<\frac{\max\bigl\{2+\sqrt2\,a, \pi(1+a)/2\bigr\}(1-x)^{1/2}}{a+(1+x)^{1/2}},\quad x\in(0,1).
\end{multline}
For $a\ge2\sqrt2\,$, the inequality~\eqref{Carlson-Arccos-thm-3-ineq} reverses on $(0,1)$.
\par
Moreover, the constants $2+\sqrt2\,a$ and $\frac\pi2(1+a)$ in~\eqref{Carlson-Arccos-thm-3-ineq} and~\eqref{Carlson-Arccos-thm-3-ineq-max} are the best possible.
\end{thm}

\section{Remarks}

Before proving our theorems, we give several remarks on them as follows.

\begin{rem}
The left-hand side inequality in~\eqref{Carlson's-ineq-arccos} and the double inequality~\eqref{Carlson-Arccos-ineq-2} are the special case $a=2\sqrt2\,$ of the double inequality~\eqref{Carlson-Arccos-thm-3-ineq-max}. This shows that Theorem~\ref{Carlson-Arccos-thm-2} and Theorem~\ref{Carlson-Arccos-thm-3} sharpen and generalize the left-hand side inequality in~\eqref{Carlson's-ineq-arccos}.
\end{rem}

\begin{rem}
It is easy to verify that the function $a\mapsto\frac{1+a}{a+(1+x)^{1/2}}$ is increasing and the function $a\mapsto\frac{2+\sqrt2\,a}{a+(1+x)^{1/2}}$ is decreasing. Therefore, the sharp inequalities deduced from~\eqref{Carlson-Arccos-thm-3-ineq} are
\begin{multline}\label{Carlson-Arccos-thm-3-ineq-sharp-1}
\frac{\pi^2(1-x)^{1/2}}{2\bigl[2(\pi-2)+(4-\pi)(1+x)^{1/2}\bigr]}
<\arccos x \\ <\frac{2\bigl[2\bigl(2-\sqrt2\,\bigr)+\bigl(\sqrt2\,-1\bigr)\pi\bigr](1-x)^{1/2}} {2(\pi-2)+(4-\pi)(1+x)^{1/2}}
\end{multline}
and
\begin{equation}\label{Carlson-Arccos-thm-3-ineq-sharp-2}
\frac{\pi\bigl(1+2\sqrt2\,\bigr)(1-x)^{1/2}}{2\bigl[2\sqrt2\,+(1+x)^{1/2}\bigr]} >\arccos x >\frac{6(1-x)^{1/2}}{2\sqrt2\,+(1+x)^{1/2}}
\end{equation}
on $(0,1)$.
\par
Furthermore, it is not difficult to see that the double inequalities~\eqref{Carlson-Arccos-thm-3-ineq-sharp-1} and~\eqref{Carlson-Arccos-thm-3-ineq-sharp-2} are not included each other.
\end{rem}

\begin{rem}
Let
$$
h_x(a)=\frac{1-2/a^2}{a+(1+x)^{1/2}}
$$
for $\frac{2(\pi-2)}{4-\pi}<a<2\sqrt2\,$ and $x\in(0,1)$. Direct calculation yields
$$
h_x'(a)=\frac{4 \sqrt{1+x}\,+6 a-a^3}{a^3 \bigl(a+\sqrt{1+x}\,\bigr)^2}
$$
which satisfies
\begin{multline*}
(2+a)\bigl(\sqrt3\,-1+a\bigr)\bigl(1+\sqrt3\,-a\bigr)=4+6 a-a^3\\
<a^3 \bigl(a+\sqrt{1+x}\,\bigr)^2h_x'(a)=4 \sqrt{1+x}\,+6 a-a^3\\
<4 \sqrt{2}\,+6 a-a^3=\bigl(a+\sqrt2\,\bigr)^2\bigl(2\sqrt2\,-a\bigr).
\end{multline*}
Accordingly,
\begin{enumerate}
\item
when $\frac{2(\pi-2)}{4-\pi}<a\le1+\sqrt{3}\,$, the function $a\mapsto h_x(a)$ is increasing;
\item
when $1+\sqrt{3}\,<a<2\sqrt2\,$, the function $a\mapsto h_x(a)$ attains its maximum
$$
\frac{4\cos^2\Bigl(\frac13\arctan\frac{\sqrt{1-x}\,}{\sqrt{1+x}\,}\Bigr)-1}{4\Bigl[2\sqrt{2}\, \cos\Bigl(\frac13\arctan\frac{\sqrt{1-x}\,} {\sqrt{1+x}\,}\Bigr)+\sqrt{1+x}\,\Bigr] \cos^2\Bigl(\frac13\arctan\frac{\sqrt{1-x}\,}{\sqrt{1+x}\,}\Bigr)}
$$
at the point
$$
2\sqrt{2}\, \cos\biggl(\frac13\arctan\frac{\sqrt{1-x}\,}{\sqrt{1+x}\,}\biggr).
$$
\end{enumerate}
As a result, the sharp inequalities deduced from~\eqref{Carlson-Arccos-thm-3-ineq-max} are
\begin{equation}\label{further-sharp-ineq-1}
\frac{8\bigl[1-2/\bigl(1+\sqrt{3}\,\bigr)^2\bigr](1-x)^{1/2}}{1+\sqrt{3}\,+(1+x)^{1/2}}<\arccos x <\frac{\pi\bigl(2-\sqrt{2}\,\bigr)(1-x)^{1/2}}{{4-\pi}+\bigl(\pi-2\sqrt{2}\,\bigr)(1+x)^{1/2}}
\end{equation}
and
\begin{equation}\label{further-sharp-ineq-2}
\frac{2\Bigl[4\cos^2\Bigl(\frac13\arctan\frac{\sqrt{1-x}\,}{\sqrt{1+x}\,}\Bigr)-1\Bigr](1-x)^{1/2}} {\Bigl[2\sqrt{2}\, \cos\Bigl(\frac13\arctan\frac{\sqrt{1-x}\,} {\sqrt{1+x}\,}\Bigr)+\sqrt{1+x}\,\Bigr] \cos^2\Bigl(\frac13\arctan\frac{\sqrt{1-x}\,}{\sqrt{1+x}\,}\Bigr)}<\arccos x
\end{equation}
on $(0,1)$.
\end{rem}

\begin{rem}
By the famous software M\textsc{athematica}~7.0 and standard computation, we show that
\begin{enumerate}
\item
the inequality~\eqref{further-sharp-ineq-2} includes the right-hand side inequality in~\eqref{Carlson-Arccos-thm-3-ineq-sharp-2} and the left-hand side inequality in~\eqref{further-sharp-ineq-1};
\item
the left-hand side inequality~\eqref{Carlson-Arccos-thm-3-ineq-sharp-1} and the inequality~\eqref{further-sharp-ineq-2} are not included each other;
\item
the upper bound in~\eqref{further-sharp-ineq-1} is better than those in~\eqref{Carlson-Arccos-thm-3-ineq-sharp-1} and~\eqref{Carlson-Arccos-thm-3-ineq-sharp-2}.
\end{enumerate}
In conclusion, we obtain the following best and sharp double inequality
\begin{multline}
\frac{\pi\bigl(2-\sqrt{2}\,\bigr)(1-x)^{1/2}}{{4-\pi} +\bigl(\pi-2\sqrt{2}\,\bigr)(1+x)^{1/2}}>\arccos x\\
>\max\Biggl\{\frac{2\bigl[4\lambda^2(x)-1\bigr](1-x)^{1/2}} {\bigl[2\sqrt{2}\, \lambda(x)+\sqrt{1+x}\,\bigr] \lambda^2(x)}, \frac{\pi^2(1-x)^{1/2}}{2\bigl[2(\pi-2)+(4-\pi)(1+x)^{1/2}\bigr]}\Biggr\}
\end{multline}
for $x\in(0,1)$, where
\begin{equation}
\lambda(x)=\cos\biggl(\frac13\arctan\frac{\sqrt{1-x}\,} {\sqrt{1+x}\,}\biggr),\quad x\in(0,1).
\end{equation}
\end{rem}

\begin{rem}
This paper is a revised version of the preprint~\cite{Carlson-Arccos-further.tex-arXiv}.
\end{rem}

\begin{rem}
The approach used in this paper to prove Theorem~\ref{Carlson-Arccos-thm-2} and
Theorem~\ref{Carlson-Arccos-thm-3} has been utilized in~\cite{Carlson-Arccos.tex,
Oppeheim-Sin-Cos.tex, Shafer-ArcSin.tex, Shafer-Fink-ArcSin.tex, Shafer-ArcTan.tex-arXiv, Shafer-ArcTan.tex,
Arc-Hyperbolic-Sine-alter.tex} to establish similar monotonicity and inequalities
related to the arc sine, arc cosine and arc tangent functions.
\end{rem}

\section{Proofs of Theorem~\ref{Carlson-Arccos-thm-2} and Theorem~\ref{Carlson-Arccos-thm-3}}

Now we are in a position to verify our theorems.

\begin{proof}[Proof of Theorem~\ref{Carlson-Arccos-thm-2}]
Straightforward differentiation yields
\begin{align*}
F_a'(x)&=\frac{{\sqrt{1-x^2}\bigl(a\sqrt{x+1}\,+2\bigr)}}{2 (x-1)^2 (x+1)} \biggl[\frac{2 (x-1) \bigl(a\sqrt{x+1}\,+x+1\bigr)}{\sqrt{1-x^2}\,\bigl(a\sqrt{x+1}\,+2\bigr)}+\arccos x\biggr]\\
&\triangleq\frac{{\sqrt{1-x^2}\bigl(a\sqrt{x+1}\,+2\bigr)}}{2 (x-1)^2 (x+1)}G_a(x),
\end{align*}
and
\begin{align*}
G_a'(x)&=\frac{\bigl(a^2\sqrt{x+1}\,-ax-a-4\sqrt{x+1}\,\bigr)\sqrt{1-x}\,} {(1+x)\bigl(a\sqrt{x+1}\,+2\bigr)^2}\\
&\triangleq\frac{H_a(x)\sqrt{1-x}\,} {(1+x)\bigl(a\sqrt{x+1}\,+2\bigr)^2}
\end{align*}
\par
It is clear that only if $a\not\in\bigl(-2,-\sqrt2\,\bigr)$ the denominators of $G_a'(x)$ and $G_a(x)$ do not equal zero on $(0,1)$ and that the function $H_a(x)$ has two zeros
$$
a_1(x)=\frac{x+1-\sqrt{x^2+18 x+17}\,}{2\sqrt{x+1}\,}\quad\text{and}\quad a_2(x)=\frac{x+1+\sqrt{x^2+18 x+17}\,}{2\sqrt{x+1}\,}
$$
whose derivatives are
$$
a_1'(x)=\frac{\sqrt{x^2+18 x+17}\,-x-1}{4\sqrt{(1+x)(x^2+18 x+17)}\,}>0
$$
and
$$
a_2'(x)=\frac{1+x+\sqrt{x^2+18 x+17}\,}{4\sqrt{(1+x)(x^2+18 x+17)}\,}>0
$$
with
\begin{gather*}
\lim_{x\to0^+}a_1(x)=\frac{1-\sqrt{17}\,}2,\quad \lim_{x\to1^-}a_1(x)=-\sqrt2\,,\\
\lim_{x\to0^+}a_2(x)=\frac{1+\sqrt{17}\,}2,\quad \lim_{x\to1^-}a_2(x)=2\sqrt2\,.
\end{gather*}
Since the functions $a_1(x)$ and $a_2(x)$ are strictly increasing on $(0,1)$, the following conclusions can be derived:
\begin{enumerate}
\item
When $a\le-2<\frac{1-\sqrt{17}\,}2<-\sqrt2\,$ or $a\ge2\sqrt2\,$, the function $H_a(x)$ and the derivative $G_a'(x)$ are always positive on $(0,1)$, and so the function $G_a(x)$ is strictly increasing on $(0,1)$. From
\begin{equation}\label{G-a-limits}
\lim_{x\to0^+}G_a(x)=\frac{(\pi-4)a+2(\pi-2)}{2(a+2)}\quad\text{and}\quad \lim_{x\to1^-}G_a(x)=0,
\end{equation}
it follows that the functions $G_a(x)$ and $F_a'(x)$ are negative, and so the function $F_a(x)$ is strictly decreasing on $(0,1)$.
\item
When $-\sqrt2\,\le a\le\frac{1+\sqrt{17}\,}2$, the function $H_a(x)$ and the derivative $G_a'(x)$ are negative on $(0,1)$, and so the function $G_a(x)$ is strictly decreasing on $(0,1)$. From~\eqref{G-a-limits}, it is obtained that the function $G_a(x)$ and the derivative $F_a'(x)$ are positive. So the function $F_a(x)$ is strictly increasing on $(0,1)$.
\item
When $\frac{1+\sqrt{17}\,}2<a<2\sqrt2\,$, the functions $H_a(x)$ and $G_a'(x)$ have a unique zero which is the unique maximum point of $G_a(x)$. From~\eqref{G-a-limits}, it is deduced that
\begin{enumerate}
\item
if $\frac{1+\sqrt{17}\,}2<a\le\frac{2(\pi-2)}{4-\pi}$, the functions $G_a(x)$ and $F_a'(x)$ are positive, and so the function $F_a(x)$ is strictly increasing on $(0,1)$.
\item
if $\frac{2(\pi-2)}{4-\pi}<a<2\sqrt2\,$, the functions $G_a(x)$ and $F_a'(x)$ have a unique zero which is the unique minimum point of the function $F_a(x)$ on $(0,1)$.
\end{enumerate}
\end{enumerate}
\par
On the other hand, the derivative $F_a'(x)$ can be rearranged as
\begin{align*}
F_a'(x)&=\frac{{\sqrt{1-x^2}}}{2 (x-1)^2 (x+1)} \biggl[\frac{2 (x-1)
\bigl(a\sqrt{x+1}\,+x+1\bigr)}{\sqrt{1-x^2}\,}+\bigl(a\sqrt{x+1}\,+2\bigr)\arccos
x\biggr]\\
&\triangleq\frac{{\sqrt{1-x^2}}}{2 (x-1)^2 (x+1)} Q_a(x),
\end{align*}
with
\begin{gather*}
Q_a'(x)=\frac{\arccos x}{2\sqrt{x+1}\,} \biggl(a-\frac{4\sqrt{1-x}\,}{\arccos x}\biggr)
\triangleq \frac{\arccos x}{2\sqrt{x+1}\,}[a-P(x)],\\
\begin{split}
P'(x)&=\frac{2(x+1)}{\sqrt{x+1}\, \sqrt{1-x^2}\,(\arccos x)^2}
\biggl[\frac{2\sqrt{1-x^2}\,}{x+1}-\arccos x\biggr]\\
&\triangleq\frac{2(x+1)}{\sqrt{x+1}\, \sqrt{1-x^2}\,(\arccos x)^2}R(x)
\end{split}
\end{gather*}
and
$$
R'(x)=\frac{x-1}{(x+1) \sqrt{1-x^2}}\,<0.
$$
From $\lim_{x\to1^-}R(x)=0$ and the decreasingly monotonic property of $R(x)$, we
obtain that $R(x)>0$, and so the function $P(x)$ is strictly increasing. Since
$$
\lim_{x\to0^+}P(x)=\frac8\pi\quad \text{and} \quad\lim_{x\to1^-}P(x)=2\sqrt2\,,
$$
the function $Q_a(x)$ is strictly decreasing (or increasing, respectively) with respect
to $x\in(0,1)$ for $a\le\frac8\pi$ (or $a\ge2\sqrt2\,$, respectively). By virtue of
$\lim_{x\to1^-}Q_a(x)=0$, it follows that
\begin{enumerate}
\item
if $a\le\frac8\pi$, the function $Q_a(x)$ is positive on $(0,1)$;
\item
if $a\ge2\sqrt2\,$, the function $Q_a(x)$ is negative on $(0,1)$.
\end{enumerate}
These implies that the function $F_a(x)$ is strictly increasing for
$a\le\frac8\pi<\frac{2(\pi-2)}{4-\pi}$ and strictly decreasing for $a\ge2\sqrt2\,$. The
proof of Theorem~\ref{Carlson-Arccos-thm-2} is complete.
\end{proof}

\begin{proof}[Proof of Theorem~\ref{Carlson-Arccos-thm-3}]
Easy calculation gives
\begin{equation*}
\lim_{x\to0^+}F_a(x)=\frac\pi2(1+a)\quad\text{and}\quad \lim_{x\to1^-}F_a(x)=2+\sqrt2\,a.
\end{equation*}
By the monotonicity of $F_a(x)$ procured in Theorem~\ref{Carlson-Arccos-thm-2}, it follows that
\begin{enumerate}
\item
if $a\le\frac{2(\pi-2)}{4-\pi}$, then
$$
\frac\pi2(1+a)<F_a(x)<2+\sqrt2\,a
$$
on $(0,1)$, which can be rearranged as the inequality~\eqref{Carlson-Arccos-thm-3-ineq};
\item
if $a\ge2\sqrt2\,$, the inequality~\eqref{Carlson-Arccos-thm-3-ineq} is reversed;
\item
if $\frac{2(\pi-2)}{4-\pi}<a<2\sqrt2\,$, the function $F_a(x)$ has a unique minimum, so
$$
F_a(x)<\max\Bigl\{\frac\pi2(1+a), 2+\sqrt2\,a\Bigr\}
$$
on $(0,1)$, which is equivalent to the right-hand side inequality~\eqref{Carlson-Arccos-thm-3-ineq-max}.
\par
Furthermore, the minimum point $x_0\in(0,1)$ of the function $F_a(x)$ satisfies
$$
\arccos x_0=\frac{2(1-x_0)\bigl(a\sqrt{x_0+1}\,+x_0+1\bigr)} {\sqrt{1-x_0^2}\,\bigl(a\sqrt{x_0+1}\,+2\bigr)}
$$
and so
$$
F_a(x_0)=\frac{2\bigl(a+\sqrt{x_0+1}\,\bigr)\bigl(a\sqrt{x_0+1}\,+x_0+1\bigr)} {\sqrt{1+x_0}\,\bigl(a\sqrt{x_0+1}\,+2\bigr)}
\triangleq\frac{2(a+u)^2} {au+2}\ge8\biggl(1-\frac2{a^2}\biggr),
$$
where $u=\sqrt{1+x_0}\,\in\bigl(1,\sqrt2\,\bigr)$. The left-hand side inequality in~\eqref{Carlson-Arccos-thm-3-ineq-max} follows.
\end{enumerate}
The proof of Theorem~\ref{Carlson-Arccos-thm-3} is complete.
\end{proof}

\section{An open problem}

Finally, we propose the following open problem.

\begin{open}\label{Carlson-Arccos-open}
For real numbers $\alpha$, $\beta$ and $\gamma$, let
\begin{equation}
F_{\alpha,\beta,\gamma}(x)=\frac{\gamma+(1+x)^\beta}{(1-x)^\alpha}\arccos x, \quad x\in(0,1).
\end{equation}
Find the ranges of the constants $\alpha$, $\beta$ and $\gamma$ such that the function $F_{\alpha,\beta,\gamma}(x)$ is monotonic on $(0,1)$.
\end{open}

\end{document}